\documentclass[12pt]{amsart}
\usepackage{amssymb, amsmath}
\usepackage{relsize}

\usepackage{mathrsfs}
\usepackage[T1]{fontenc}    
\usepackage{hyperref}       
\usepackage{url}            
\usepackage{booktabs}       
\usepackage{amsfonts}       
\usepackage{nicefrac}       
\usepackage{microtype}      
\usepackage{lipsum}
\usepackage{graphicx}

\usepackage{cleveref}  

\usepackage{pst-node}
\usepackage{tikz-cd} 

\usepackage{mathtools}

\usepackage{verbatim}
\graphicspath{ {./images/} }
\usepackage{xcolor}

\usepackage{pst-node}
\usepackage{tikz-cd} 
\usetikzlibrary{arrows,automata}

\def\multiset#1#2{\ensuremath{\left(\kern-.3em\left(\genfrac{}{}{0pt}{}{#1}{#2}\right)\kern-.3em\right)}}

\numberwithin{equation}{section}

\newcommand{\bc}{\mathbf{c}}
\newcommand{\bd}{\mathbf{d}}

\newcommand{\m}{\mathbf{m}}
\newcommand{\n}{\mathbf{n}}
\newcommand{\br}{\mathbf{r}}

\newcommand{\bk}{\mathbf{k}}
\newcommand{\bi}{\mathbf{i}}

\newcommand{\cA}{\mathcal{A}}
\newcommand{\cC}{\mathcal{C}}

\newcommand{\cL}{\mathcal{L}}
\newcommand{\cM}{\mathcal{M}}

\newcommand{\CI}{\mathscr{I}}

\newcommand{\C}{\mathbb{C}}
\newcommand{\K}{\mathbb{K}}
\newcommand{\N}{\mathbb{N}}
\newcommand{\Q}{\mathbb{Q}}
\newcommand{\Z}{\mathbb{Z}}
\newcommand{\eps}{\varepsilon}
\newcommand{\om}{\omega}
\DeclareMathOperator{\FI}{FI}
\DeclareMathOperator{\id}{id}
\DeclareMathOperator{\im}{im}
\DeclareMathOperator{\Mon}{Mon}
\DeclareMathOperator{\one}{\mathbf{1}}
\DeclareMathOperator{\Supp}{Supp}

\newtheorem{thm}{Theorem}[section]

\newtheorem{cor}[thm]{Corollary}
\newtheorem{lemma}[thm]{Lemma}
\newtheorem{prop}[thm]{Proposition}
\newtheorem{question}[thm]{Question}

\theoremstyle{definition}
\newtheorem{defn}[thm]{Definition}
\newtheorem{ex}[thm]{Example}
\newtheorem{rem}[thm]{Remark}

\title{Equivariant Hilbert series for Hierarchical Models}

\author{Aida~Maraj}
\address{Department of Mathematics, University of Kentucky, 715 Patterson Office Tower, Lexington, KY 40506-0027, USA}
\email{aida.maraj@uky.edu} 
 
\author{Uwe~Nagel}
\address{Department of Mathematics, University of Kentucky, 715 Patterson Office Tower, Lexington, KY 40506-0027, USA}
\email{uwe.nagel@uky.edu} 
 
\begin{document}

\begin{abstract}
Toric ideals to hierarchical models are invariant under the action of a product of symmetric groups. Taking the number of factors, say $m$, into account, we introduce and study invariant filtrations and their equivariant Hilbert series. We present a condition that guarantees that the equivariant Hilbert series is a rational function in $m+1$ variables with rational coefficients. Furthermore we give explicit formulas for the rational functions with coefficients in a number field and an algorithm for determining the rational functions with rational coefficients. A key is to construct finite automata that recognize languages corresponding to invariant filtrations. 
\end{abstract}

\thanks{Both authors were partially supported by Simons Foundation grants \#317096 and \#636513.
}


\maketitle 



\section{Introduction}

Hierarchical models are used in algebraic statistics to determine dependencies among random variables (see, e.g., \cite{Su}). Such a model  is determined by a simplicial complex and the number of states each random variable can take.  The Markov basis to any hierarchical model corresponds to a generating set of an 
associated toric ideal, see \cite{DS}. This toric ideal is rather symmetric, that is, it is invariant under the action of a product of symmetric groups. The number of minimal generators of the toric ideals 
grows rapidly when the number of states of the considered random variables increases. However, 
the Independent Set Theorem (see \Cref{thm:independent set}) shows that the symmetry can be 
leveraged to describe, for a fixed simplicial complex,  simultaneously the generating sets and thus 
Markov bases for all  numbers of states of the random variables. The conceptional proof of this 
result by Hillar and Sullivant \cite{HS} introduces the notion of an $S_{\infty}$-invariant filtration. 
Informally, this is a sequence $(I_n)_{n \in \N}$ of compatible ideals $I_n$ in polynomial rings 
$R_n$ whose number of variables increases with $n$ and where each $I_n$ is invariant under the action 
of a symmetric group that permutes the variables of $R_n$.  To such a filtration, the second author 
and R\"omer \cite{NR} introduced an equivariant Hilbert series in order to analyze simultaneously  
quantitative properties of the ideals in the filtration. It is a formal power series in two variables and they showed that it is rational with rational coefficients \cite[Theorem 7.8]{NR}. 

The variables occurring in the elements of a toric ideal to a hierarchical model can naturally be grouped 
into $m$ sets of variables, where $m$ is the number of random variables. Permuting the variables in any one of these groups gives a group action that leaves the ideal invariant. 
This suggests the introduction of an $S_{\infty}^m$-invariant 
filtration (see \Cref{def:filtration}). For $m = 1$ it specializes to the filtrations mentioned above. 
Every $S_{\infty}^m$-invariant filtration  gives naturally rise to an equivariant Hilbert series defined 
as a formal power series in $m+1$ variables (see \Cref{def:equiv Hilb}). Our main result gives a condition guaranteeing that this power series is a rational function in $m+1$ variables with rational coefficients (see \Cref{thm:main}). Furthermore, we present two methods to determine this rational function. One approach is more special and produces an explicit rational function, but with coefficients in a suitable extension field of the rational numbers (see \Cref{prop:explict computation}). The other approach is much more general and gives directly a formula for the rational function with rational coefficients. It determines the equivariant Hilbert series as the generating function of a regular language (see Section 5). 

The remaining part of this paper is organized as follows. In Section 2, we discuss the symmetry of toric ideals to hierarchical models and introduce $S_{\infty}^m$-invariant filtrations. Their 
equivariant Hilbert series in $m+1$ variables are studied in Section 3. Our main result about such 
Hilbert series is stated as \Cref{thm:main}. We reduce its proof to a special case in that section, but 
postpone the argument for the special case to the following section. In Section 4 we use regular languages and finite automata to establish the special case. The idea is to encode the monomials 
that determine the Hilbert series by a language. We then construct a deterministic finite  
automaton that recognizes this language. Thus, the language is regular. Using a suitable weight 
function we then show that the corresponding generating function of the language is essentially the 
desired Hilbert series. Since generating functions of regular languages are rational this completes 
the argument of our main result. Furthermore, using the finite automaton that describes a regular language, there is an algorithm that determines the generating function of the language explicitly as a rational function with rational coefficients. This is explained and illustrated in Section 5. We also describe in that section a more limited direct approach that gives an explicit formula for the rational function, but with coefficients in a number field. 


\section{Symmetry and Filtrations} 

After reviewing needed concepts and notation we introduce $S_{\infty}^m$-invariant filtrations in this section. 

Throughout this paper we use $\N$ and $\N_0$  to denote the set of positive integers and the  set of non-negative integers, respectively. For any $q \in \N$, we set $[q]=\{1,2,\ldots,q\}$,  and so $[0]= \emptyset$. We use $\# T$ to denote the number of elements in a finite set $T$. 

A \emph{hierarchical model} $\mathcal{M} = \mathcal{M} (\Delta,\mathbf{r})$  with $m$ parameters is given by a collection $\Delta = \{ F_1,F_2, \cdots, F_q\}$ of non-empty subsets $F_j\subset [m]$  with $\bigcup_{j\in[q]} F_j=[m]$ and a vector $\mathbf{r}=(r_1,r_2,\cdots,r_m) \in \N^m$.  Each parameter corresponds to a random variable, and $r_i$ denotes the number of values parameter $i$ can take. We refer to  $\mathbf{r}$ as the \emph{vector of states}. Every set $F_j$  indicates a dependency among the parameters corresponding to its vertices. Thus, we may assume that no $F_j$ is contained in some $F_i$ with $ i \neq j$, and refer to the sets $F_j$ as \emph{facets}.

Diaconis and Sturmfels \cite{DS} pioneered the use of algebraic methods in order to study statistical models. We need some notation. For any subset $F=\{i_1,i_2,\dots,i_s \}\subset [m]$, we write 
\[
\mathbf{r}_F=(r_{i_1},r_{i_2},\dots, r_{i_s}) \in \N^s \; \text{ and } \;  [\mathbf {r}_F]=[r_{i_1}]\times [r_{i_2}]\times  \dots \times [r_{i_s}] \subset \N^s. 
\] 
In particular, $[\mathbf{r}_{[m]}]=[\mathbf{r}] \subset \N^m$. 
Given a field $\mathbb K$ and a hierarchical model $\mathcal{M}=\mathcal{M}(\mathbf{r},\Delta)$,  consider the following ring homomorphism:
\begin{equation} 
    \label{eq:def homo}
\begin{split}
 \Phi_{\mathcal{M}} \colon \boldmath R_{\br}=\mathbb K[x_{\mathbf{i}} \mid  \mathbf{i} \in [\mathbf{r}] ] & \longrightarrow 
 \boldmath S_{\mathcal{M}}= \mathbb K [y_{j,\mathbf{i}_{F_j}} \mid   F_j \in \Delta, \mathbf{i}_{F_j} \in [\mathbf{r}_{F_j}]],  \\
 x_{\mathbf{i}} & \longmapsto \prod_{F_j\in \Delta}y_{j, \mathbf{i}_{ F_j}} 
 \end{split}
 \end{equation}
The  kernel of this homomorphism, denoted $I_{\mathcal{M}}$, is called the \emph{toric ideal}  to the hierarchical model ${\mathcal{M}}$. We also refer to $R_{\br}/I_{\cM}$ as the \emph{coordinate ring} of the model $\cM$. 

In the simplest cases explicit sets of generators of such ideals are known. We use the standard partial order $\le$ on $\Z^s$ given by  $\bi = (i_1,\ldots,i_s) \le \mathbf{j} = (j_1,\ldots,j_s)$ if $i_1 \le j_1,\ldots, i_s \le j_s$. If $q = 1$ then $\Phi_{\cM}$ is an isomorphism, and so $I_{\cM}$ is zero. 

\begin{ex} 
    \label{exa:two facets}
Let $q = 2$, i.e., $\Delta = \{F_1, F_2\}$.   

(i) Suppose first that $F_1$ and $F_2$ are disjoint. Possibly permuting the positions of the entries of a vector $\bi \in [\br] = [\br_{F_1 \cup F_2}]$, we write $x_{\bi_{F_1}, \bi_{F_2}}$ instead of $x_{\bi}$. This corresponds to a bijection $[\br_{F_1 \cup F_2}] \to [\br_{F_1}] \times [\br_{F_2}]$.  Using this notation,  a  generating set of $I_{\mathcal{M}}$ is (see, e.g.,\cite{DG} and \cite{DS})   
\begin{align*}
\hspace{2em}&\hspace{-2em} 
G(\mathcal{M}(\mathbf{r},\{F_1, F_2\})) = \\
& \{ x_{\mathbf{i}_{F_1},\mathbf{i}_{F_2}}x_{\mathbf{i'}_{F_1},\mathbf{i'}_{F_2}}-x_{\mathbf{i}_{F_1},\mathbf{i'}_{F_2}}x_{\mathbf{i'}_{F_1},\mathbf{i}_{F_2}} \:  \mid \: \mathbf{i}_{F_1} < \mathbf{i'}_{F_1}\in [\mathbf{r}_{F_1}],  \ \mathbf{i}_{F_2}< \mathbf{i'}_{F_2} \in [\mathbf{r}_{F_2}]\}
\end{align*}

In the special case, where $m = 2$ and, say, $F_1 = \{1\}, F_2 = \{2\}$, this set becomes 
\[
\{ x_{i_1, i_2} x_{i'_1, i'_2} - x_{i_1, i'_2} x_{i'_1, i_2} \mid 1 \le i_1 \le i'_1 \le r_1, \ 1 \le i_2 \le i'_2 \le r_2 \}, 
\]
which is the set of  $2 \times 2$ minors of a generic $r_1 \times r_2$ matrix with entries 
$x_{i_1, i_2}$. The image of the map  $\Phi_{\mathcal{M}}$ in this case is known in algebraic geometry as the coordinate ring of the Segre product  $\mathbb{P}^{r_1-1} \times \mathbb{P}^{r_2-1}$ whose homogeneous ideal is $I_{\cM}$.   

(ii) Consider now the general case, where $F_1$ and $F_2$ are not necessarily disjoint. Note that $[m]$ is the disjoint union of $F_1 \setminus F_2, F_2 \setminus F_1$ and $F_1 \cap F_2$. Thus,  possibly permuting the positions of the entries of $\bi \in [r]$ as above, 
we write $x_{\mathbf{i}_{F_1\setminus F_2}, \bi_{F_1 \cap F_2} ,\mathbf{i}_{F_2\setminus F_1}}$ for $x_{\bi}$. Fixing a vector $\bc \in [\br_{F_1 \cap F_2}]$, we define a set $G^{\mathbf{c}}(\mathcal{M}(\mathbf{r}_{[m]\setminus F_1\cap F_2},\{F_1\setminus F_2, F_2\setminus F_1\})$ whose elements are 
\[
x_{\mathbf{i}_{F_1\setminus F_2}, \mathbf{c},\mathbf{i}_{F_2\setminus F_1}}x_{\mathbf{i'}_{F_1\setminus F_2}, \mathbf{c},\mathbf{i'}_{F_2\setminus F_1}}-x_{\mathbf{i'}_{F_1\setminus F_2}, \mathbf{c},\mathbf{i}_{F_2\setminus F_1}}x_{\mathbf{i}_{F_1\setminus F_2}, \mathbf{c},\mathbf{i'}_{F_2\setminus F_1}}, 
\]
where 
\[
\mathbf{i} _{F_1\setminus F_2} < \mathbf{i'}_{F_1\setminus F_2} \in [\mathbf{r}_{F_1\setminus F_2}] \text{ and }  \mathbf{i}_{F_2\setminus F_1} < \mathbf{i'} _{F_2\setminus F_1} \in [\mathbf{r}_{F_2\setminus F_1}]. 
\]
The collection 
\[
G(\mathcal{M}(\mathbf{r},\{F_1,F_2\}))=\bigcup\limits_{\mathbf{c}\in[\mathbf{r}_{F_1\cap F_2}]}G^{\mathbf{c}}(\mathcal{M}(\mathbf{r}_{[m]\setminus F_1\cap F_2},\{F_1\setminus F_2, F_2\setminus F_1\}))
\]
is a generating set for the ideal $I_{\mathcal{M}(\mathbf{r},\{F_1,F_2\})}$ (see \cite{D}). 
\end{ex}

Even in the simple cases of \Cref{exa:two facets}, the number of minimal generators of a toric ideal $I_{\cM}$ is large if the entries of $\br$ are large. However, many of these generators have similar shape. This can be made precise using symmetry. 

Indeed, denote by $S_n$ the symmetric group in $n$ letters. Set $S_{[\br]} = S_{r_1} \times S_{r_2} \times \cdots \times S_{r_m}$. This group acts on the polynomial ring $R_{\br}$ by permuting the indices of its variables, that is, 
\[
(\sigma_1,\ldots,\sigma_m) \cdot x_{\bi} = x_{(\sigma_1 (i_1),\ldots,\sigma_m (i_m))}. 
\]
 It is well-known that toric ideals have minimal generating sets consisting of binomials. Thus, the 
definition of the homomorphism $\Phi_{\cM}$ in \eqref{eq:def homo} implies that the ideal 
$I_{\cM}$ is $S_{[\br]}$-invariant, that is, $\sigma \cdot f \in I_{\cM}$ whenever $\sigma \in S_{[\br]}$ 
 and $f \in I_{\cM}$. In some cases, this invariance can be used to obtain  all minimal generators of $I_{\cM}$ 
 from a subset by using symmetry. For example, in the special case $m= q=2$, 
 $F_1 = \{1\}, F_2 = \{2\}$ with $r_1, r_2 \ge 2$, the set $G(\mathcal{M}(\mathbf{r},\{F_1, F_2\}))$ 
 can be obtained from 
 \[
 x_{1,1} x_{2,2} - x_{1,2} x_{2,1} 
 \]
 using the action of $S_{r_1} \times S_{r_2}$. Note that this is true for every vector $\br = (r_1, r_2)$. There is a vast generalization of this observation using the concept of an invariant filtration. 
 
 The symmetric group  $S_n$ is naturally embedded into  $S_{n+1}$ as the stabilizer of $\{n+1\}$. Using this construction componentwise, we get an embedding of $S_{[\br]}$ into $S_{[\br']}$ if $\br \le \br'$. Set 
 \[
 S_{\infty}^m = \bigcup_{\br \in \N^m}S_{[\br]}. 
 \]
 
 \begin{defn}
     \label{def:filtration}
 An \emph{$S_{\infty}^m$-invariant filtration} is a family $(I_{\br})_{\br \in \N^m}$ of ideals $I_{\br} \subset R_{\br}$  such that every ideal $I_{\br}$ is $S_{[\br]}$-invariant and, as subsets of $R_{\br'}$, 
 \[
 S_{[\br']} \cdot I_{\br} \subset I_{\br'} \quad \text{ whenever } \br \le \br'. 
 \]
\end{defn}

Note that fixing $\Delta$, the ideals $(I_{\cM (\Delta, \br)})_{\br \in \N_0^m}$ form an $S_{\infty}^m$-invariant filtration. It is useful to extend these ideas. 

\begin{rem}
    \label{rem:restricted action}
Let $T$ be any non-empy subset of $[m]$. For vectors $\br \in \N^m$, we want to fix the entries in positions supported at $T$, but vary the other entries. To this end write $(\br_{[m] \setminus T}, \br_T)$ instead of $\br$. 

Fix a vector $\bc \in \N^{m \setminus \#T}$. 
Let $I_{\br} \subset R_{\br}$ be an $S_{\infty}^m$-invariant filtration. 
Restricting  $S_{[\br]}$ and its action to components supported at $T$, we get an $S_{\infty}^{\# T}$-invariant filtration of ideals $I_{\br_T} = I_{\bc, \br_T} \subset R_{\bc, \br_T}$ with $\br_T \in \N^{\#T}$. 
\end{rem}

Note that this idea applies to the ideals $I_{\cM (\Delta, \br)}$ with fixed $\Delta$. 
We can now state the mentioned extension of the example given above \Cref{def:filtration}. It is called Independent Set theorem and has been established by Hillar and Sullivant  in \cite[Theorem 4.7]{HS} (see also \cite{DEKL}). 

\begin{thm}
   \label{thm:independent set}
Fix $\Delta$ and consider  a subset $T \subset [m]$ such that $\#(F_j \cap T) \le 1$ for every 
$j \in [q]$. Assume the number of states of every parameter $j \in [m] \setminus T$ is fixed, and 
consider the hierarchical models  $\cM (\Delta, \br_T) = \cM(\Delta, (\bc, \br_T))$, where 
$\bc \in \N^{m - \#T}$. Then the ideals $I_{\cM (\Delta, \br_T)}$ 
form an $S_{\infty}^{\# T}$-invariant filtration 
$\mathscr{I}_{\Delta,\mathbf{r}_{[m] \setminus T}} = (I_{\cM (\Delta, \br_T)})_{\br_T \in \N^{\#T}}$, that 
is, there is some $\bd \in \N^{\#T}$ such that $S_{[\br_T]} \cdot I_{\cM (\Delta, \bd)}$ generates in 
$R_{\bc, \br_T}$ the ideal $I_{\cM (\Delta, \br_T)}$  whenever $\br_T \ge \bd$.
\end{thm}

In other words, this result says that a generating set of the ideal $I_{\cM(\Delta, \br)}$ can be obtained from a fixed finite minimal generating set of $I_{\cM(\Delta, (\bc, \bd))}$ by applying suitable permutations whenever the number of states of every parameter in $[m] \setminus T$ is large enough. 

\Cref{thm:independent set} is not true without an assumption on the set $T$ (see \cite[Example 4.3]{HS}). 

\begin{rem}
An $S_{\infty}^m$-invariant filtration can also be described using a categorial framework. Indeed, if $m=1$ this approach has been used in \cite{NR2} to study also sequences of modules by using the category $\FI$, whose objects are finite sets and whose morphisms are injections.  This approach can be extended to any $m \ge 1$ using the category $\FI^m$ (see, e.g., \cite{LY} in the case of modules over a fixed ring). For conceptional simplicity we prefer to use invariant filtrations in this paper. 
\end{rem}


\section{Equivariant Hilbert Series} 

In order to study asymptotic properties of ideals in an $S_{\infty}$-invariant filtration, an equivariant 
Hilbert series was introduced in \cite{NR}. Here we study an extension of this concept for 
$S_{\infty}^m$-invariant filtrations. 

We begin by recalling some basic facts. Let $I$ be a homogeneous ideal in a polynomial ring $R$ in finitely many variables over some field $\K$. We will always assume that any variable has degree one. Thus,  $R/I = \oplus_{j \ge 0} [R/I]_j$ is a standard graded $\K$-algebra. Its Hilbert series is the formal power series 
\[ 
H_{R/I} (t)=\sum\limits_{j \ge 0}\dim_{\mathbb{K}} [R/I]_j t^j. 
\]
By Hilbert's theorem (see, e.g., \cite[Corollary 4.1.8]{BH}), it is rational and can be uniquely written as 
\[ 
H_{R/I} (t)= \dfrac{g(t)}{(1-t)^{\dim R/I}} 
\]
with $g(t)\in \mathbb{Z}[t]$ and $g(1) > 0$, unless $I = R$. The number $g(1)$ is called the \emph{degree} of $I$. 

\begin{defn}
       \label{def:equiv Hilb}
The \emph{equivariant Hilbert series} of an $S_{\infty}^m$-invariant filtration $\CI = (I_{\br})_{\br \in \N^m}$ of ideals  $I_{\br} \subset R_{\br}$ is the formal power series in variables $s_1,\ldots,s_m, t$ 
\begin{align*}
equivH_{\mathscr{I}} (s_1,\ldots,s_m, t) & = \sum_{\br \in \N^m} H_{R_{\br}/I_{\br}} (t) \cdot s_1^{r_1} \cdots s_m^{r_m} \\
& = \sum_{\br \in \N^m} \sum_{j \ge 0} \dim_{\K}  [R_{\br}/I_{\br}]_j  \cdot s_1^{r_1} \cdots s_m^{r_m} t^j.  
\end{align*}
\end{defn}

If $m = 1$, that is, $\CI$ is an $S_{\infty}$-invariant filtration, the Hilbert series of $\CI$ is always rational by \cite[Theorem 7.8]{NR} or \cite[Theorem 4.3]{KLS}. For $m \ge 1$, one can also consider another formal power series by focussing on components whose degree is on the diagonal of $\N^m$. This gives 
\begin{align*}
\sum_{r \ge 1} H_{R_{(r,\ldots,r)}/I_{(r,\ldots,r)}} (t) \cdot s^r. 
\end{align*}
It is open whether this formal power series is  rational if $m \ge 2$, even if the ideals are trivial. 

\begin{ex}
Let $m = 2$ and consider the filtration $\CI = (I_{\br})$, where every ideal $I_{\br}$ is zero. 
Since the ring $R_{(r_1,r_2)}$ has dimension $r_1 r_2$, one obtains
\begin{align*}
equivH_{\mathscr{I}} (s_1,s_2, t) & = \sum_{(r_1, r_2) \in \N^2} H_{R_{(r_1, r_2)}} (t) \cdot s_1^{r_1} s_2^{r_2} 
= \sum_{(r_1, r_2) \in \N^2}   \frac{1}{(1 - t)^{r_1 r_2}} \cdot s_1^{r_1} s_2^{r_2}  \\
& = \sum_{r_1 \ge 1} \left [ -1 + \dfrac{(1-t)^{r_1}}{(1-t)^{r_1}-s_2} s_1^{r_1} \right ]. 
\end{align*}
We do not know if this is a rational function in $s_1, s_2$ and $t$. However, if one considers the more standard Hilbert series with $r = r_1 = r_2$ one gets 
\[
\sum_{r \ge 0}  H_{R_{(r, r)}} (t) \cdot s^r = \sum_{n\geq 1} \dfrac{1}{(1-t)^{r^2}} \cdot s^r. 
\]
This is not a rational function because the sequence $\left (\frac{1}{(1-t)^{r^2}} \right)_{r \in \N}$ does not satisfy a finite linear recurrence relation with coefficients in $\Q(t)$. 
\end{ex}

For the remainder of this section we restrict ourselves to considering ideals of hierarchical models 
$\cM (\Delta, \br)$. As pointed out in \Cref{rem:restricted action}, for any subset 
$T \neq \emptyset$ of $[m]$, these ideals give rise to $S_{\infty}^{\# T}$-invariant filtrations. To 
study their equivariant Hilbert series, it is convenient to simplify notation. We may assume that 
$T = \{m- \# T + 1,\ldots,m \}$ and fix  the entries of $\br$ in positions supported on $[m] \setminus  T$, that 
is, we fix $\bc \in \N^{m - \#T}$ and  set $\n = (n_1,\ldots,n_{m - \# T}) = \br_T$ for 
$\br \in \N^m$ to obtain $\br = (\bc, \n)$. We write $\cM (\Delta, \n)$ instead of  
$\cM(\Delta, (\bc, \n))$ and denote the resulting $S_{\infty}^{m - \# T}$-invariant filtration 
$(I_{\cM (\Delta, \n)})_{\n \in \N^{m - \# T}}$ by $\mathscr{I}_{\Delta,\mathbf{r}_{[m] \setminus T}}$, as in the Independent Set theorem. Its equivariant Hilbert series is 
\[
equivH_{\mathscr{I}_{\Delta,\mathbf{r_{[m] \setminus T}}}} (s_1,s_2,\dots,s_{\# T}, t)=\sum\limits_{\n \in \N^{\# T}} H_{R_{(\bc, \n)}/ I_{\cM (\Delta, \n)}} (t) \cdot s_1^{n_1} \cdots   s_{\# T}^{n_{\# T}}.
\]

The Independent Set theorem (\Cref{thm:independent set}) guarantees stabilization of the filtration. This suggests the following problem. 

\begin{question}
        \label{q:independent so rational}
If $T \subset [m]$ satisfies $\# (F \cap T) \le 1$ for every facet $F$ of $\Delta$, is then the 
equivariant Hilbert series of $\mathscr{I}_{\Delta,\mathbf{r}_{[m] \setminus T}}$ rational?\end{question} 

The answer is affirmative if $T$ consists of exactly one element. 

\begin{prop}
     \label{prop: one-element set T}
If $\# T = 1$, then the 
equivariant Hilbert series of $\mathscr{I}_{\Delta,\mathbf{r}_{[m] \setminus T}}$ is rational. 
\end{prop}

\begin{proof}
The assumption means $T=\{m\}$ and $\br = (\bc, n)$ with $\bc \in \N^{m-1}$ and $n \in \N$. Set $c = c_1  \cdots c_{m-1}$ and fix a bijection 
\[
\psi \colon [\bc] = [c_1] \times \cdots \times [c_{m-1}] \to [c]. 
\]
For every $n \in \N$, it induces a ring isomorphism 
\begin{align*}
R_{(\bc, n)} =\mathbb{K}[x_{\mathbf{i}, j} \mid (\mathbf{i}, j) \in [\bc]\times [n]]  \longrightarrow &\mathbb{K}[x_{i, j} \mid  (i, j) \in [c]\times [n]]=R'_n \\
x_{\mathbf{i}, j} & \mapsto  x_{\psi{(\bi)}, j}. 
\end{align*}
This isomorphism maps every  ideal $I_{\cM (\Delta, n)}$ corresponding to the model $\cM (\Delta, (\bc, n))$ onto an $S_n$-invariant ideal $I_n$. In particular, the rings 
$R_{(\bc, n)}/I_{\cM (\Delta, n)}$ and $R'_n/I_n$ have the same Hilbert series and the family 
$(I_n)_{n \in \N}$ is an $S_{\infty}$-invariant filtration. Thus, its equivariant Hilbert series is rational by \cite[Theorem 7.8]{NR} or \cite[Theorem 4.3]{KLS}. 
\end{proof} 

Our main result in this section describes further cases in which the answer to \Cref{q:independent so rational} is affirmative. 

\begin{thm} 
          \label{thm:main}
The equivariant Hilbert series of $\mathscr{I}_{\Delta,\mathbf{r}_{[m]\backslash T}}$  is a rational function with rational coefficients if
\begin{enumerate}
    \item  $F_i\cap F_j=\emptyset$ for any  distinct $F_i,F_j \in \Delta$.
    \item $|F\cap T|\leq 1$ for any $F\in \Delta$.
\end{enumerate}
\end{thm} 

This results applies in particular to the independence model, where it takes an attractive form. 

\begin{ex}  
     \label{ex:indep model}
A hierarchical model describing $m$ independent parameters is called \emph{independence model}. Its collection of facets is $\Delta=\{\{1\},\{2\},\dots,\{m\}\}$. Thus, we may apply 
\Cref{thm:main} with any subset $T$ of $[m]$. Using $T = [m]$, we show in \Cref{ex:equiv HS} 
 below that 
\begin{align*}
equivH_{\mathscr{I}_{\Delta,\mathbf{r}_{[m] \setminus T}}} (s_1,s_2,\ldots,s_{m}, t) 
& =\sum\limits_{\n \in \N^{m}} H_{R_{\n}/ I_{\cM (\Delta, \n)}} (t) \cdot s_1^{n_1} \cdots   s_{m}^{n_{m}}\\
& =
\dfrac{s_1\cdots s_m}{(1-s_1)\cdots(1-s_m)-t}. 
\end{align*}
\end{ex}

The proof of \Cref{thm:main} will be given in two steps. First we show that it is enough to prove the result in a special case where every facet consists of two elements. Second, we use regular languages to show the desired rationality in the following section. 

In the remainder of this section we establish the reduction step.

\begin{lemma} 
      \label{lem:reduction}
Consider a collection $\Delta = \{F_1,\ldots,F_q\}$ on vertex set $[m]$ and a subset $T$ of $[m]$ satisfying 
\begin{enumerate}
    \item  $F_i\cap F_j=\emptyset$ for any $F_i,F_j \in \Delta$.
    \item $|F\cap T| =1$ for any $F\in \Delta$. 
\end{enumerate}
Then there is a collection $\Delta' = \{F'_1,\ldots,F'_q\}$ on vertex set $[m']$ consisting of two 
element facets and also satisfying conditions (1) and (2) with the property that, for every 
$\bc \in \N^{m-\# T}$ there is some $\bc' \in \N^{m' - \#T}$ such that the filtrations corresponding to 
the models $\cM(\Delta, (\bc, \n))$ and $\cM(\Delta', (\bc', \n))$ with $\n \in \N^{\# T}$ have the same equivariant Hilbert series. 
\end{lemma}

\begin{proof}
The assumptions imply that $T$ must have $q$ elements. We may assume that every facet in $\Delta$ has at least two elements. Indeed, if $F \in \Delta$ has only one element then we may replace $F$ by the union $F'$ of $F$ and a new vertex. Assigning to the parameter corresponding to the new vertex exactly one possible state gives a new model whose coordinate ring has the same Hilbert series as the original model. 

Given such a hierarchical  model $\cM_{\n} = \cM(\Delta, (\bc, \n))$ on vertex set $[m]$, we will construct a new hierarchical model 
$\cM'_{\n} = \cM(\Delta', (\bc', \n))$ on $m' = 2 q$ vertices that has the same Hilbert series. The new vertex set is the disjoint union of the $q$ vertices in 
$F_j \cap T$ with $j \in [q]$ and a set $V$ of $q$ other vertices, say $V = [q]$. For $j \in [q]$, set 
$F'_j = \{j\} \cup (F_j \cap T)$. Thus, the sets $F'_j$ are pairwise disjoint because $F_1,\ldots,F_q$ have this property, and 
each $F'_j$ has exactly two elements. In particular,  $\Delta' = \{F'_1,\dots,F'_q\}$ and $T$ satisfy conditions (1) and (2). 

Now let $c'_j=\prod\limits_{e\in F_j \setminus T} c_e = \# [\bc_{F_j \setminus T}]$ be the number of states of the parameter corresponding to vertex $j \in F'_j$. Furthermore, for every $j \in [q]$, let the parameter corresponding to vertex 
$F'_j \cap T$ have the same number of states as $F_j \cap T$ has in $\cM_{\n}$. This completes the definition of   a new  hierarchical model 
$\cM'_{\n} = \cM(\Delta', (\bc', \n))$. The passage form $\cM_{\n}$ to $\cM'_{\n}$ is illustrated in an example below. 

\includegraphics[scale=0.2]{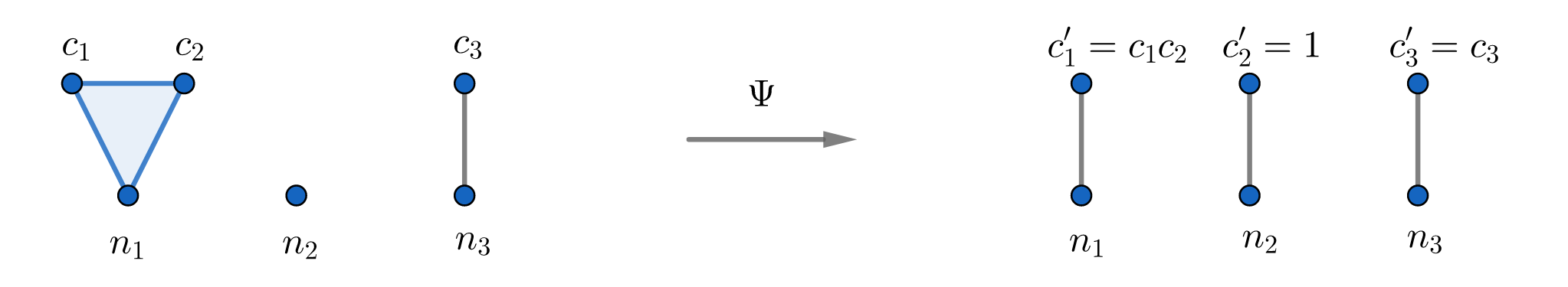}
\begin{eqnarray*}
{\small  \Delta=\{124,5,36\}, r=(c_1,c_2,c_3,n_1,n_2,n_3) } & \longrightarrow & {\small \Delta'=\{14,25,36\}, r'=(c'_1,1,c'_3,n_1,n_2,n_3)  }
\end{eqnarray*}


Varying $\n \in \N^q$, the ideals $I_{\cM'_{\n}}$ form an $S_{\infty}^q$-invariant filtration. Thus, to establish the assertion 
it is enough to prove that  for every $\mathbf{n}\in \mathbb{N}^q$,  the quotient  rings $R_{\mathbf{n}}\slash I_{\cM_{\n}}$ and  $R'_{\mathbf{n}}\slash I_{\cM'_{\n}}$ are isomorphic.

For every $F_j\in\Delta$, the sets $[\mathbf{c}_{F_j \setminus T}]$ and $[c'_j]$ have the same finite cardinality. Choose a bijection 
\[
\psi_{j}:[\mathbf{c}_{F_j \setminus T}] \longrightarrow [c'_j]. 
\]
These choices determine two further bijections:
\begin{equation}
      \label{eq:bijection 1}
(\psi_{1},\dots,\psi_{q},\id_{[\mathbf{n}}]) \colon [\mathbf{c}_{F_1 \setminus T}]\times\dots \times[\mathbf{c}_{F_q \setminus T}]\times [\mathbf{n}] \longrightarrow    [c'_1]\times\dots\times[c'_q]\times[\mathbf{n}]
\end{equation}
and 
\begin{equation}
      \label{eq:bijection 2}
(\psi_{j},\id_{[n_j]}) \colon [\mathbf{c}_{F_j \setminus T}]\times[n_j] \longrightarrow    [c'_j]\times[n_j].  
\end{equation}
Bijection \eqref{eq:bijection 1} induces the following isomorphism of polynomial rings
\begin{align*}
\hspace{19em}&\hspace{-19em}  
 \Psi \colon  R_{(\bc,\n)} =\mathbb{K}[x_{\mathbf{i}_{F_1 \setminus T},\dots,\mathbf{i}_{F_q \setminus T},\mathbf{k}} \mid \mathbf{i}_{F_q \setminus T} \in [\mathbf{c}_{F_q \setminus T}], \mathbf{k} \in [\mathbf{n}]] \\
 &  \longrightarrow     \mathbb{K}[x_{i_1,\dots,i_q,\mathbf{k}} \mid i_j\in[c'_j], \mathbf{k}\in [\mathbf{n}]]=R'_{\n} \\
 x_{\mathbf{i}_{F_1 \setminus T},\dots,\mathbf{i}_{F_q \setminus T},\mathbf{k}} & \mapsto x_{\psi_1(\mathbf{i}_{F_1 \setminus T}),\dots,\psi_q (\mathbf{i}_{F_q \setminus T}),\mathbf{k}}. 
 \end{align*}
 Similarly,   Bijection \eqref{eq:bijection 2} induces an isomorphism of polynomial rings
\begin{align*}
\hspace{15em}&\hspace{-15em}  
\Psi' \colon S_{\n}=\mathbb{K}[y_{j,\mathbf{i}_{F_j \setminus T},k_j} \mid 1\leq j\leq q, \mathbf{i}_{F_j \setminus T}\in [\mathbf{c}_{F_j \setminus T}] , k_j \in [n_j]] \\
&\longrightarrow     \mathbb{K}[y_{j,i_j,k_j} \:|\: 1\leq j\leq q,  i_j\in [c_j] , k_j \in [n_j]]=S'_{\n} \\
y_{j,\mathbf{i}_{F_j \setminus T},k_j} & \mapsto y_{j,\psi_j(\mathbf{i}_{F_j \setminus T}),k_j}. 
\end{align*}
We claim that the following diagram is commutative: 
\begin{equation}
     \label{eq:diagram 3}
\begin{tikzcd}R_{(\bc, \n)} \arrow{r}{\Phi_{\cM}} \arrow[swap]{d}{\Psi} & S_{\n}  \arrow{d}{\Psi' } \\
R'_{\n} \arrow{r}{\Phi_{\cM'}}&  S'_{\n}\end{tikzcd}
\end{equation}
Indeed, it suffices to check this for variables variables. In this case commutative is shown by the diagram:   
\[ 
\begin{tikzcd}
x_{\mathbf{i}_{F_1 \setminus T},\dots,\mathbf{i}_{F_q \setminus T},\mathbf{k}} \arrow[r, mapsto, "\Phi_{\cM}"]{} \arrow[d,swap, mapsto, "\Psi"]{} & \prod\limits_{j=1}^q y_{j,\mathbf{i}_{F_j \setminus T},k_j}  \arrow[d, mapsto, "\Psi'"]{} \\%
x_{\psi_1(\mathbf{i}_{F_1 \setminus T}),\dots,\psi_q(\mathbf{i}_{F_q \setminus T}),\mathbf{k}} \arrow[r, mapsto, "\Phi_{\cM'}"]{}& \prod\limits_{j=1}^q y_{j,\psi_1(\mathbf{i}_{F_q \setminus T}),k_j}
\end{tikzcd}
\]
Since $\Psi$ and $\Psi'$ are isomorphisms  commutativity of Diagram \eqref{eq:diagram 3}  implies that 
$\im (\Phi) \cong \im(\Phi')$, which concludes the proof. 
\end{proof}

We also need the following result. 

\begin{prop}
      \label{prop:second step}
Let $\CI = \{I_{\n}\}_{\n\in\N^q}$ be the $S_{\infty}^q$-invariant  filtration corresponding to hierarchical models  $\cM(\Delta, (\bc, \n))$ with 
$\Delta$ consisting of $q$ $2$-element disjoint facets $F_1,\ldots,F_q$, each meeting $T$ in exactly one vertex. 
Then the equivariant Hilbert series of $\CI$ is a rational function in $s_1,\ldots,s_q, t$ with rational coefficients. 
\end{prop}

This will be shown in the following section. Assuming the result, we  complete the argument for establishing  \Cref{thm:main}. 

\begin{proof}[Proof of  \Cref{thm:main}] 
Let $\nu$ be the number of facets in $\Delta$ whose intersection with $T$ is empty. We use induction on $\nu \ge 0$. If $\nu =0$ the claimed rationality follows by combining \Cref{lem:reduction} and \Cref{prop:second step}. 

Let $\nu \ge 1$. We may assume that $F_1 \cap T = \emptyset$ and that vertex 1 is in $F_1$. By assumption, it has $c_1$ states. Set $\tilde{\n} = (n_1, \n)$, $\tilde{\bc} = (c_2,\ldots, c_{\# T})$ and 
$\tilde{T} = T \cup \{1\}$. Then the hierarchical models $\tilde{\cM} (\Delta, (\tilde{\bc}, \tilde{\n})$ give rise to a filtration $\tilde{\CI} = \CI_{\Delta, \br_{[\m] \setminus \tilde{T}}}$. By induction on $\nu$, it has a rational equivariant Hilbert series. Since $equivH_{\CI}$ is obtained by evaluating 
$\dfrac{1}{c_1!} \dfrac{\partial^{c_1} equivH_{\tilde{\CI}}}{\partial s_1^{c_1}}$ at $s_1 = 0$, it follows that also $equivH_{\CI}$ is rational. 
\end{proof}


\section{Regular Languages}

The goal of this section is to establish \Cref{prop:second step}. We adopt its notation. 

Fix $\bc \in \N^q$. As above, we write $x_{\bi,\mathbf{k}}$ for $x_{i_1,\dots,i_q, k_1,\dots,k_q}$, where $(\bi, \bk) = (i_1,\dots,i_q, k_1,\dots,k_q) \in [\bc] \times [\n] \subset \N^{2 q}$. Thus, $y_{j, \bi_{F_j}, \bk_{F_j}}$ is simply $y_{j, i_j, k_j}$. 
For any $\n \in \N^q$,  the  homomorphism associated to the model $\cM_{\n} = \cM(\Delta, (\bc, \n))$ is 
\begin{align*}
\Phi_{\n} \colon R_{\n} = \K[x_{\bi, \bk}  \mid (\bi, \bk) \in [\bc] \times [\n] ] 
& \to \K[y_{j, i_j, k_j} \mid j \in [q], i_j \in c_j, k_j \in [n_j] ] = S_{\n} \\
x_{\bi, \bk} & \longmapsto \prod_{j=1}^q  y_{j, i_j, k_j}.
\end{align*}
Set 
\[
A_\n = \im \Phi_{\n} = \K \left [ \prod_{j=1}^q  y_{j, i_j, k_j}  \mid  i_j \in c_j, k_j \in [n_j] \right ].
\]
We denote the set of monomials in $A_\n$ and $S_\n$  by $\Mon (A_\n)$ and $\Mon (S_\n)$, respectively.  Define $\Mon (A)$ as the disjoint union of the sets $\Mon (A_\n)$ with $\n \in \N^q$ and similarly $\Mon (S)$.  Our next goal is to show that the elements of $\Mon (A)$ are in bijection to the words of a suitable formal language. 

Consider a set  
\[
\Sigma=\{ \zeta_{\mathbf{i}}, \tau_j \mid \mathbf{i}\in [\mathbf{c}],  j \in [q]\} 
\]
with  $q + \prod\limits_{j=1}^q c_j$ elements. Let $\Sigma^*$ be the free monoid on $\Sigma$. A 
\emph{formal language} with words in the alphabet $\Sigma$ is a subset of $\Sigma^*$. We refer to the 
elements of $\Sigma$ as letters. The empty word is denoted by $\eps$. 

In order to compare subsets of $\Sigma^*$ with $\Mon (A)$ we 
need suitable maps. For $j \in [q]$, define a shift operator
$
T_j \colon \Mon (S) \to \Mon (S) 
$ 
by 
\[
T_j(y_{l, i, k}) = \begin{cases}
y_{l, i, k+1} & \text{ if } l = j; \\
y_{l, i, k} & \text{ if } l \neq j,  
\end{cases}
\]
extended  multiplicatively to $\Mon (S)$.  Define a map $\textbf{m} \colon \Sigma^{\star} \rightarrow \Mon (S)$ inductively using the three rules 
\[
\text{(a) } \textbf{m} (\epsilon)=1,  \quad
\text{(b) } \textbf{m} (\zeta_{\bi} w )=\prod\limits_{j=1}^q y_{j,i_j,1}\textbf{m}(w),  \quad
\text{(c) } \textbf{m} (\tau_j w)=T_j(\textbf{m}(w)), 
\]
where $w\in \Sigma^*$. In particular, this gives  $\textbf{m} (\zeta_{\bi}) = \Phi_\n (x_{\bi, \one})$ for any 
$\n \in \N^q$, where $\one$ is the $q$-tuple whose entries are all equal to 1. 

\begin{ex}
If $c_1 =c_2 = q = 2$, one  has $\Sigma=\{ \zeta_{1,1},\zeta_{1,2},\zeta_{2,1}, \zeta_{2,2}, \tau_1,\tau_2\}$, and, for any $\n \ge (2, 3)$,  
\begin{align*}
\mathbf{m}(\tau_1\tau_2\zeta_{1,2}\tau_{2}\zeta_{1,1}\tau_1) 
& = T_1(T_2( y_{1,1,1} y_{2,2,1}T_2 (y_{1,1,1} y_{2,1,1}T_1(1)))) \\
& = T_1(T_2( y_{1,1,1} y_{2,2,1}   y_{1,1,1} y_{2,1,2})) \\
& = y_{1,1,2} y_{2,2,2}   y_{1,1,2} y_{2,1,3} \\
& = \Phi_\n (x_{(1,2),(2,2)}) \Phi_\n (x_{(1,1),(2,3)}).
\end{align*}
\end{ex}

The map $\m$ is certainly not injective because the variables $y_{j,i,k}$ commute. For example, if $q = 2$ one has $\mathbf{m}(\tau_1\tau_2)=\mathbf{m}(\tau_2\tau_1)$ and 
$\mathbf{m}(\zeta_{2,1}\zeta_{1,2})=\mathbf{m}(\zeta_{1,2}\zeta_{2,1})=\mathbf{m}(\zeta_{1,1}\zeta_{2,2})$ 
and $\mathbf{m}(\tau_1\zeta_{1,2}\tau_2\zeta_{2,1})=\mathbf{m}(\tau_1 \zeta_{2,2} \tau_2 \zeta_{1,1})$. 
Thus, we introduce a suitable subset of $\Sigma^*$. 

\begin{defn}
      \label{def:language L}
Let  $\mathcal{L}$ be the set of words in $\Sigma^*$ that satisfy the following conditions:
\begin{enumerate}
\item Every substring $\tau_i \tau_j$ has $i \le j$. 

\item In every substring with no $\tau_j$, if $\zeta_{\bi}$ occurs to the left of some $\zeta_{\bi'}$, then the $j$-th entry of $\bi$ is less than or equal to the $j$-th entry of $\bi'$.
\end{enumerate}
\end{defn}

To avoid triple subscripts below, we denote the $j$-th entry of a $q$-tuple $\bk_l$ by $k_{(l, j)}$, that is, we write
\[
\mathbf{k}_l=(k_{(l,1)},k_{(l,2)},\dots,k_{(l,q)}) \in \N^q. 
\]
Using multi-indices,  we write $\tau^a$ for $\tau_1^{a_1}\tau_2^{a_2}\dots \tau_q^{a_q}$ with $a=(a_1,a_2,\dots,a_q)$. A string consisting only of $\tau$-letters can be written as $\tau^\bk$ if and only if it satisfies Condition (1) in  \Cref{def:language L}. With this notation, one gets immediately the following explicit description of the words in $\cL$.

\begin{lemma}
     \label{lem:explicit words}
The elements of the formal language $\mathcal{L}$ are precisely the words of the form 
\[
\tau^{\mathbf{k}_1}\zeta_{\mathbf{i}_1}\tau^{\mathbf{k}_2}\zeta_{\mathbf{i}_2}\dots \tau^{\mathbf{k}_d}\zeta_{\mathbf{i}_d}\tau^{\mathbf{k}_{d+1}}, 
\] 
where  $\mathbf{i}_1,\dots,\bi_d\in [\mathbf{c}]$, $\mathbf{k}_1,\dots,\bk_{d+1}\in \mathbb{N}_0^q$, and $i_{(l-1, j)} \le i_{(l, j)}$ whenever $k_{(l, j)} = 0$ for some $(l, j)$ with   $2\leq l \leq d$ and $j \in [q]$. 
\end{lemma}

The following elementary observation is useful. 

\begin{lemma}
     \label{lem:normal form monomials}
Every monomial in $\Mon (A)$ can be uniquely written as a string of variables such that one has, the variable in any position $l$ is of the form $y_{j,i_j,k_j}$ with $j = l \!\!\!\mod q$ and, 
for each $j \in [q]$, if a variable  $y_{j,i_j,k_j}$ appears to the left of $y_{j,i'_j,k'_j}$, then 
either $k_j<k'_j$ or $k_j=k'_j$ and $i_j\leq i'_j$. 
\end{lemma}

\begin{proof}
If for some $j$, two variables $y_{j,i_j,k_j}$ and $y_{j,i'_j,k'_j}$ appearing in a monomial do not 
satisfy the stated condition, then swap their positions. Repeating this step as long as needed results 
in a string meeting the requirement. It is unique, because the given condition induces an order on 
the variables $y_{j,i,k}$ with fixed $j$. In the desired string, for each fixed $j$, the variables 
$y_{j, i, k}$ occur in this order when one reads the string from left to right. 
\end{proof}

We illustrate the above argument. 

\begin{ex}
Let $q=2$. To simplify notation write $y_{j k}$ instead of $y_{1, j, k}$ and $z_{j k}$ instead of 
$y_{2, j , k}$. Then one gets, for example, 

\begin{tikzpicture} [align=center]
\path  (0, 0)   node[]   (B)  {$    y_{22}z_{21}y_{14}z_{11}y_{31}z_{21}$};
\path   (4, 0.5)   node[]  (C) {$y_{22}y_{14}y_{31}$ };
\path    (4,-0.5)   node[]  (D) {$z_{21}z_{11}z_{21}$};
\path     (7, 0.5)   node[]  (E) {$\color{blue}{y_{31}y_{22}y_{14}}$};
\path     (7,-0.5)   node[]  (F) {$\color{red}{z_{11}z_{21}z_{21}}$};
\path     (11, 0)   node[]   (G)  { $\color{blue}{y_{31}}\color{red}{z_{11}} \color{blue}{y_{22}}\color{red}{z_{21}} {\color{blue}{y_{14}}\color{red}{z_{21}}}$ };
\path   (B) edge node{}(C);
\path   (B) edge node{}(D);
\path   (C) edge node{}(E);
\path   (D) edge node{}(F);
\path   (F) edge node{}(G);
\path   (E) edge node{}(G);
\end{tikzpicture}
\end{ex}

We observed above that the map $\m$ sends each letter $\zeta_{\bi}$ to the monomial $\Phi_\n (x_{\bi, \one})$. It follows that $\m (\Sigma^*)$ is a subset of $\Mon (A)$. In fact, one has the following result. 

\begin{prop}
       \label{prop:bijection}
For any $\n \in \N_0^q$, denote by $\cL_\n$ the set of words in $\cL$ in which, for each $j \in [q]$, the letter $\tau_j$ occurs precisely $n_j$ times. Then $\m$ induces  for every $\n \in \N_0^q$ a bijection 
\[
\m_\n \colon \cL_\n \to \Mon (A_{\n + \one}), \; w \mapsto \m (w).  
\]
\end{prop}

\begin{proof}
The definition of $\m$ readily implies $\m (w) \in \Mon (A_{\n + \one})$ if $w \in \cL_\n$. 

First we show that $\m_\n$ is surjective. Let $m \in \Mon (A_{\n + \one})$ be any monomial. Its 
degree is $d q$ for some $d \in \N_0$. By \Cref{lem:normal form monomials}, $m$ can be written as 
\[
m=\prod\limits_{l=1}^d \left (\prod\limits_{j=1}^qy_{j,i_{(l,j)},k_{(l,j)}} \right) = \prod\limits_{l=1}^d \Phi_\n (x_{\bi_l, \bk_l})
\]
such that, for each $j \in [q]$, one has 
\[
1 \le k_{(1,j)}\leq \dots \leq k_{(d,j)}\leq n_j+1
\]
and 
\[
i_{(l-1, j)} \le i_{(l, j)} \quad \text{ if  $k_{(l, j)} = 0$ for some $l$}. 
\]
The first condition implies that all the $q$-tuples $\bk_1 - \one, \bk_2 - \bk_1,\ldots, \bk_d - \bk_{d-1}$ and $\n+ \one - \bk_d$ are in $\N_0^q$. Hence the string 
\[
w=\tau^{\mathbf{k}_1-\one}\zeta_{\mathbf{i}_1}\tau^{\mathbf{k}_2-\mathbf{k}_1}\zeta_{\mathbf{i}_2} \dots \tau^{\mathbf{k}_d-\mathbf{k}_{d-1}}\zeta_{\mathbf{i}_d} \tau^{\n+ \one - \bk_d}
\]
is defined. The two conditions together combined with  \Cref{lem:explicit words} show that in fact $m$ is in $\cL_\n$. Hence $\m (w) = m$ proves the claimed surjectivity. 

Second, we establish that $\m_\n$ is injective. Consider any two words $w, w' \in \cL_\n$ with $\m(w)=\m(w')$. We will show $w = w'$. 

Write $w$ and $w'$ as in \Cref{lem:explicit words}: 
\[
w=\tau^{\mathbf{k}_1}\zeta_{\mathbf{i}_1}\tau^{\mathbf{k}_2}\zeta_{\mathbf{i}_2}\dots \tau^{\mathbf{k}_d}\zeta_{\mathbf{i}_d}\tau^{\mathbf{k}_{d+1}} , \quad 
w'=\tau^{\mathbf{k'}_1}\zeta_{\mathbf{i'}_1}\tau^{\mathbf{k'}_2}\zeta_{\mathbf{i'}_2}\dots \tau^{\mathbf{k'}_{d'}}\zeta_{\mathbf{i'}_{d'}}\tau^{\mathbf{k'}_{d'+1}}
\]
Since $\m (w)$ has degree $d q$ and $\m (w')$ has degree $d' q$, we conclude $d = d'$. Evaluating $\m$ we obtain 
\begin{equation}
     \label{eq:compare mon}
\prod\limits_{l=1}^d(\prod\limits_{j=1}^q y_{j,i_{(l,j)},f_{(l,j)}}) = 
\prod\limits_{e=1}^{d}(\prod\limits_{j=1}^qy_{j,i'_{(l,j)},f'_{(l,j)}}), 
\end{equation}
where $f_{(l,j)}=k_{(1,j)}+\dots +k_{(l,j)}+1$ and $f'_{(l,j)}=k'_{(1,j)}+\dots +k'_{(l,j)}+1$. 
Fix any $j \in [q]$. Comparing the  third indices of the variables whose first index equals $j$ and using that every index is non-negative, we get for each $l \in [d]$, 
\[
k_{(1,j)}+\dots +k_{(l,j)} = k'_{(1,j)}+\dots +k'_{(l,j)}. 
\]
It follows that $\bk_l = \bk'_l$ for each $l \in [d]$. Since $w$ and $w'$ are in $\cL_{\n}$, we have  
$\bk_{d+1} = \n-(\bk_1+\bk_2+\dots +\bk_d)$ and an analogous equation for $\bk'_{d+1}$, which 
gives $\bk_{d+1} = \bk'_{d+1}$. 

It remains to show $\bi_l=\bi'_l$ for every $l \in [d]$. Fix any $j \in [q]$. If for some $l$ 
there is only one variable of the form $y_{j, \mu, f_{(l, j)}}$ with $\mu \in [c_j]$  that divides $\m (w)$, this implies $i_{(l, j)} = i'_{(l, j)} = \mu$, as desired. 
Otherwise, there is a maximal interval of consecutive indices $k_{(l, j)}$ that are equal to zero, that is, there any integers $a, b$ such that  $1 \le a \le b \le d$ and 
\begin{itemize}
\item $k_{(l, j)} = 0$ if $a \le l \le b$, 

\item $k_{(a-1, j)} > 0$, unless $a = 1$, and 

\item $k_{(b+1, j)} > 0$, unless $b = d$. 
\end{itemize} 
Thus, the number of variables of the form $y_{j, \mu, f_{(l, j)}}$ that divide $\m (w)$ is $b - a + 2$ if $a \ge 2$ and $b-a+1$ if $a = 1$. Considering these variables, \Cref{lem:explicit words} gives  
\[
i_{(a-1, j)} \le i_{(a, j)} \le \cdots \le i_{(b, j)} \; \text{ and } \; i'_{(a-1, j)} \le i'_{(a, j)} \le \cdots \le i'_{(b, j)}, 
\]
where $i_{(a-1, j)}$ and $i'_{(a-1, j)}$ are omitted if $a = 1$.  Using \eqref{eq:compare mon}, it now  follows that $i_{(l, j)} = i'_{(l, j)}$ whenever $a-1 \le l \le b$, unless $a =1$. If $a = 1$, the latter equality is true whenever $a \le l \le b$. 

Applying the latter argument to any interval of consecutive zero indices $k_{(l, j)}$, we conclude 
$i_{(l, j)} = i'_{(l, j)}$ for every $l \in [d]$. This completes the argument. 
\end{proof}

Our next goal is to show that $\cL$ is a regular language. By \cite[Theorems 3.4 and 3.7]{HU}, this is equivalent to proving that $\cL$ is recognizable by a finite automaton. Recall that a \emph{deterministic finite 
automaton} on an alphabet $\Sigma$  is a $5$-tuple $\cA = (P, \Sigma, \delta, p_0, F)$ consisting  
of a  finite set $P$ of \emph{states},  an \emph{initial state} $p_0 \in P$, a set $F \subset P$ of 
\emph{accepting states}  and a \emph{transition map} $\delta \colon D \to P$, where $D$  is  some 
subset of $P \times \Sigma$.  We refer to  $\cA$ simply as  a finite automaton because we will consider 
only deterministic automata. The automaton $\cA$ \emph{recognizes} or \emph{accepts} a word 
$w=a_1a_2\dots a_s \in \Sigma^*$ if there is a sequence of states $r_0,r_1,\dots, r_s$ satisfying 
$r_0 = p_0$, $r_s \in F$ and 
\[
r_{j+1} = \delta (r_j, a_{j+1}) \quad \text{whenever } 0 \le j < s. 
\]
In words, the automaton starts in state $p_0$ and  transitions from state $r_j$ to a state $r_{j+1}$ based on the input $a_{j+1}$. The word $w$ is accepted if $r_s$ is an accepting state. If $\delta (p, a)$ is not defined the machine halts. 
The automaton $\cA$ \emph{recognizes a formal language} $\cL \subset \Sigma^*$ if $\cL$ is precisely the set of words in $\Sigma^*$ that are accepted by $\cA$. 

Returning to the formal language $\cL$ specified in \Cref{def:language L}, we are ready to show:  

\begin{prop}
      \label{prop:L recognizable} 
The language $\mathcal{L}$ is recognized by a finite automaton.
\end{prop}

\begin{proof} 
We need some further notation. 
We say that a  sequence $C$ of $l \ge 0$ integers $j_1,j_2,\ldots,j_l$ is an \emph{increasing chain in $[q]$} if $1 \le j_1< j_2 < \cdots < j_l \le q$. Define $\max (C)$ as the largest element $j_l$ of $C$. We put $\max (\emptyset) = 0$. 
We denote the set of increasing chains in $[q]$ by $\cC$. Thus, the cardinality of $\cC$ is $2^q$. We write $j \in C$ if $j$ occurs in the chain $C$. 
For any $\bk \in \N_0^q$, we define the sequence of indices $j$ with $k_j > 0$ as its support 
$\Supp (\bk)$. It is an element of $\cC$. For example, one has $\Supp (7, 0, 1, 5, 0) = (1, 3, 4)$. 

Now we define an automaton $\cA$ as follows: Let 
\[
P =\{p_j, p_{\bi}, p_{\bi,C,k} \mid  0\leq j\leq q,\:\bi\in[\bc],\: C \in \cC,\: k \in C  \}
\]
be the set of states, where $p_0$ is the initial state of $\cA$. Let 
\[
F=\{p_j, p_{\bi}, p_{\bi,C, k} \mid 0\leq j\leq q, \bi\in[\bc], \: C \in \cC,\: k = \max (C) \}
\]
be the set of accepting states. Furthermore, define transitions
\begin{align}
     \label{eq:trans1}
\delta (p_j, \tau_{j'}) & = p_{j'}  \text{ if }  j = 0 < j' \le q \text{ or }  1 \le j \le j' \le q, \\
     \label{eq:trans2}
\delta (p_j, \zeta_\bi) & = p_\bi \text{ if } 0 \le j \le q, \: \bi \in [\bc], \\
     \label{eq:trans3}
\delta (p_\bi, \tau_j) & = p_{\bi, C, j} \text{ if }  \bi \in [\bc], \: C \in \cC, \: j \in C, \\
     \label{eq:trans4}
\delta (p_\bi, \zeta_{\bi'}) & = p_{\bi'} \text{ if } \bi, \bi' \in [\bc], \: \bi \le \bi', \\
     \label{eq:trans5}
\delta( p_{\bi, C, j}, \tau_k) & = p_{\bi, C, k}, \text{ if }  \bi \in [\bc], \: C \in \cC, \: j \in C, \: k \text{ directly follows $j$ in $C$ or } k = j, \\ 
     \label{eq:trans6}
\delta( p_{\bi, C, j}, \zeta_{\bi'} ) & = p_{\bi'} \text{ if }  \bi, \bi' \in [\bc],  \: j = \max (C), \: i_k  \le i'_k \text{ whenever } k \notin C. 
\end{align}
If an element of $P \times \Sigma$ does not  satisfy any of the above six conditions then it is not in the domain of $\delta$. 

We claim that $\cA$ recognizes $\cL$. Indeed, let $w \in \Sigma^*$ be a word with exactly $d \ge 0$ $\zeta$-letters. We show by induction on $d$ that $w$ is recognized by $\cA$ if $w \in \cL$,  but any word in $\Sigma^* \setminus \cL$ is not accepted by $\cA$. It turns out that  $w \in \cL$ is accepted 
\begin{itemize}
\item at a state $p_j$ for some $0 \le j \le q$ if $d=0$, 

\item at a state $p_\bi$ for some  $\bi \in [c]$ if $d \ge 1$ and $w$ ends with a $\zeta$-letter, and 

\item at a state $p_{\bi, C, j}$  for some  $\bi \in [\bc], \: C \in \cC, \: j = \max (C)$ if $d \ge 1$ and $w$ ends with a $\tau$-letter. 
\end{itemize}
In particular, this explains the set of accepting states. 

Consider any word $w  \in \Sigma^*$  with exactly $d \ge 0$ $\zeta$-letters. Assume $d = 0$, that 
is, $w = \tau_{l_1} \tau_{l_2} \ldots \tau_{l_t}$. By transition rule \eqref{eq:trans1}, $\cA$  transitions 
from state $p_0$ to any state $p_j$ with $j \in [q]$ using input $\tau_j$. From any $p_j$ with
 $j \in [q]$ the automaton can transition to any state $p_{j'}$ with $j \le j' \le q$ by using input $\tau_{j'}$. Thus, $w$ is accepted by $\cA$ if and only if ${l_1} \le {l_2} \le \cdots \le {l_t}$, that is, $w \in \cL$ (see \Cref{lem:explicit words}). 
 
 Assume now that $d \ge 1$. We proceed in several steps. 
 
 (I) Assume $d = 1$ and $w$ ends with a $\zeta$-letter, that is, 
 \[
 w =  \tau_{l_1} \tau_{l_2} \ldots \tau_{l_t} \zeta_\bi  
 \]
 for some $t \ge 0$. The argument for $d = 0$ shows that  $\tau_{l_1} \tau_{l_2} \ldots \tau_{l_t}$ is  accepted if  and only if it can be written as some $\tau^\bk$. Processing input $\tau^\bk$, the automaton arrives at state $p_j$ with $j = \max (\Supp (\bk))$. Using input $\zeta_\bi$, it then transitions to $p_\bi \in F$ by Rule \eqref{eq:trans2}. Hence $w$ is accepted if and only of $w \in \cL$. 
 
(II) Let $d \ge 1$ and assume $w$ ends with a $\tau$-letter, that is, $w$ can be written as 
 \[
 w = w' \zeta_\bi   \tau_{l_1} \tau_{l_2} \ldots \tau_{l_t}  
 \]
with $t \ge 1$. Furthermore assume that $w' \zeta_\bi$ is accepted by $\cA$ in state $p_\bi$. 
We show that $w$ is accepted by $\cA$ if and only if $w = w'  \zeta_\bi \tau^\bk$ for some $\bk \in \N_0^q$. If $w$ is recognized,  it is accepted in state $p_{\bi, C, \max (C)}$, where $C = \Supp (\bk)$. 

Indeed, let 
$C \in \cC$ be the chain corresponding to the set $\{l_1,\ldots,l_t\}$. Processing input $\tau_{l_1}$, 
Rule \eqref{eq:trans2} yields that $\cA$ transitions to 
state $p_{\bi, C, l_1}$. If $t = 1$, then $l_1 = \max (C)$ and $w$ is accepted in $p_{\bi, C, l_1} \in F$, as claimed.  If $t \ge 2$, Rule \eqref{eq:trans5} shows that $\cA$ can transition from $p_{\bi, C, l_1}$ using 
input $\tau_{l_2}$ precisely if $l_2 \ge l_1$. If transition is possible $\cA$ gets to state $p_{\bi, C, 
l_2}$. Hence Rule \eqref{eq:trans5} guarantees that $\tau_{l_1} \tau_{l_2} \ldots \tau_{l_t}$ can be processed by $\cA$ if and only if $ \tau_{l_1} \tau_{l_2} \ldots \tau_{l_t} = \tau^\bk$ for some non-zero $\bk \in \N_0^q$. In this case $w = w' \zeta_\bi \tau^\bk$ is accepted by $\cA$ in state $p_{\bi, C, 
\max (C)}$, where $C = \Supp (\bk)$. 
 
 (III) Assume now $w \in \Sigma^*$ ends with a $\zeta$-letter, that is, $w$ is of the form 
 \[
 w = w'  \tau_{l_1} \tau_{l_2} \ldots \tau_{l_t} \zeta_\bi,  
 \]
 where $w' \in \cL$ is either empty or ends with a $\zeta$-letter and $t \ge 0$. We show  by induction on $d \ge 1$ that $w$ is recognized by $\cA$ if and only if $w \in \cL$. 
In this case, $w$ is accepted in state $p_\bi$. 
 
 Indeed, if $d = 1$, i.e., $w'$ is the empty word, this has been shown in Step (I). If $d \ge 2$ write $w' = w'' \zeta_{\bi'}$. If $w'$ is not accepted by $\cA$, then so is $w$. Furthermore, the induction hypothesis gives $w' \notin \cL$, which implies $w \notin \cL$. 
 
If $w' = w'' \zeta_{\bi'}$ is recognized by $\cA$ the induction hypothesis yields $w' \in \cL$ and $w'$ is accepted in state $p_{\bi'}$.  
 Step (II) shows that $w'' \zeta_{\bi'}  \tau_{l_1} \tau_{l_2} \ldots \tau_{l_t}$ is accepted by $\cA$ if and only if it can be written as $w'' \zeta_{\bi'} \tau^\bk$ for some $\bk \in \N_0^q$, and so   
 \[
 w = w'' \zeta_{\bi'} \tau^\bk \zeta_\bi.   
\] 
We consider two cases. 

\emph{Case 1.} Suppose $\bk$ is zero, i.e., $\Supp (\bk) = \emptyset$.  Thus, $\cA$ accepted $w'' \zeta_{\bi'} \in \cL$ in state $p_{\bi'}$. Using input $\zeta_\bi$, Rule \eqref{eq:trans4} shows that $\cA$ does not halt in $p_{\bi'}$ if and only if $\bi' \leq \bi$. By \Cref{lem:explicit words}, this is equivalent to $w = w'' \zeta_{\bi'} \zeta_\bi \in \cL$. Furthermore, if $w$ is in $\cL$ it is accepted in state $p_\bi$, as claimed. 

\emph{Case 2.} Suppose $\Supp (\bk) \neq \emptyset$. Set $C = \Supp (\bk)$. By Step (II), $w'' \zeta_{\bi'} \tau^\bk$ is accepted in state $p_{\bi', C, j}$, where $j = \max (C)$. Hence Rule 
\eqref{eq:trans6} gives that input $\zeta_\bi$ can be processed if and only if $i_l'  \le i_l$ whenever 
$l  \notin C$. By \Cref{lem:explicit words}, this is equivalent to $w = w'' \zeta_{\bi'} \tau^\bk \zeta_\bi \in \cL$. 
Moreover, if $w$ is recognized it is accepted in state $p_\bi$, as claimed.  

(IV) By Steps (I) and (III) it remains to consider the case, where $w$ ends with a $\tau$-letter, i.e.,  
$w = w' \zeta_\bi   \tau_{l_1} \tau_{l_2} \ldots \tau_{l_t}$ with $t \ge 1$. By Step (III),   $w' \zeta_\bi$ is recognized by $\cA$ if and only of $w' \zeta_\bi \in \cL$. Furthermore, if $w' \zeta_\bi \in \cL$ then it is accepted in state $p_\bi$. Hence, the assumption in Step (II) is satisfied and we conclude that $w$ is accepted if and only if $w = w' \zeta_\bi \tau^\bk$. The latter is equivalent to $w' \zeta_\bi \tau^\bk \in \cL$ because $w' \zeta_\bi$ is in $\cL$. This completes the argument. 
\end{proof}

\begin{rem}
        \label{rem:graph automaton}
Any finite automaton $\cA = (P, \Sigma, \delta, p_0, F)$ can be represented by a labeled directed graph whose vertex set is the set of states $P$. Accepting states are indicated by double circles. 
There is an edge from vertex $p$ to vertex $p'$ if there is a transition $\delta (p, a) = p'$. In that case, the edge is labeled by all $a \in \Sigma$ such that $\delta (p, a) = p'$. 
\end{rem} 

We illustrate the automata constructed in \Cref{prop:L recognizable} using such a graphical representation. 

\begin{ex}
Let $\cA$ be the automaton constructed in \Cref{prop:L recognizable} if $q = 3$ and $\bc=(1,1,1)$. 
Note the only element in $[\bc]$ is $\one = (1, 1, 1)$. 
To simplify notation, we write $\zeta$ for $\zeta_{1,1,1}$ and $p_{\one}$ for $p_{(1, 1, 1)}$. We denote the non-empty increasing chains in the interval $[3]$ as follows:  $C_1=\{1\}$, $C_2=\{2\}$, $C_3=\{3 \}$, $C_4=\{1, 2\}$, $C_5=\{1, 3\}$, $C_6=\{2, 3\}$, $C_7=\{1,2,3\}$ and write $p_{i,j}$ instead of  $p_{\mathbf{1}, C_i,j}$. Using this notation, the constructed automaton $\cA$ is represented by the following graph: 

\begin{tikzpicture}[shorten >=1pt,node distance=2.36cm,auto]
\tikzstyle{every state}=[fill={rgb:black,1;white,10}]
\node[state,initial,accepting]   (p_0)                      {$p_0$};
\node[state,accepting]           (p_1) [ below left of=p_0]     {$p_1$}; 
\node[state,accepting]           (p_2) [right   of=p_1]     {$p_2$};
\node[state,accepting]           (p_3) [right of=p_2]     {$p_3$};
\node[]           (p_4) [right of=p_3]     {};
\node[state,accepting] (p) [right of=p_4]  {$p_{\mathbf{1}}$};
\node[state] (p_5) [below left of=p]  {$p_{71}$};
\node[state] (p_6) [below  of=p]  {$p_{61}$};
\node[state] (p_7) [below right of=p]  {$p_{52}$};
\node[state] (p_8) [right of=p]  {$p_{41}$};
\node[state, accepting] (p_9) [above right of=p]  {$p_{31}$};
\node[state, accepting] (p_10) [above of=p]  {$p_{21}$};
\node[state, accepting] (p_11) [above left of=p]  {$p_{11}$};
\node[state, accepting] (p_12) [right of=p_8]  {$p_{42}$};
\node[state, accepting] (p_13) [below right of=p_7]  {$p_{53}$};
\node[state, accepting] (p_14) [below of=p_6]  {$p_{63}$};
\node[state] (p_15) [below left of=p_5]  {$p_{72}$};
\node[state, accepting] (p_16) [below left of=p_15]  {$p_{73}$};
\path[->]
(p_1)   edge  [loop below]          node {$\tau_1$} (p_1)
(p_2)   edge  [loop below]          node {$\tau_{2}$} (p_2)
(p_3)   edge  [loop below]          node {$\tau_{3}$} (p_3)
(p)   edge  [loop left]          node {$\zeta$} (p)
(p_11)   edge  [loop above]          node {$\tau_1$} (p_11) 
(p_10)   edge  [loop above]          node {$\tau_2$} (p_10) 
(p_9)   edge  [loop above]          node {$\tau_3$} (p_9) 
(p_8)   edge  [loop below]          node {$\tau_1$} (p_8) 
(p_7)   edge  [loop below]          node {$\tau_1$} (p_7) 
(p_6)   edge  [loop below]           node {$\tau_2$} (p_6) 
(p_5)   edge  [loop below]          node {$\tau_1$} (p_5) 
(p_15)   edge  [loop below]          node {$\tau_2$} (p_15)
(p_16)   edge  [loop below]          node {$\tau_3$} (p_16) (p_12)   edge  [loop below]          node {$\tau_2$} (p_12) (p_13)   edge  [loop below]          node {$\tau_3$} (p_13)
(p_14)   edge  [loop below]          node {$\tau_3$} (p_14)
(p_0)   edge          node {$\tau_1$} (p_1)
(p_0)   edge            node {$\tau_2$} (p_2)
(p_0)   edge            node {$\tau_3$} (p_3)
(p_1)   edge            node {$\tau_2$} (p_2)
(p_1)   edge   [bend right]            node {$\tau_3$} (p_3)
(p_2)   edge            node {$\tau_3$} (p_3)
(p_0)   edge         node {$\zeta$} (p)
(p_1)   edge [bend right]              node {$\zeta$} (p)
(p_2)   edge      [bend right]        node {$\zeta$} (p)
(p_3)   edge      [bend right]        node {$\zeta$} (p)
(p_2)   edge      [bend right]        node {$\zeta$} (p)
(p_9)   edge   [bend right]           node {$\zeta$} (p)
(p_10)   edge     [bend right]       node {$\zeta$} (p)
(p_11)   edge    [bend right]         node {$\zeta$} (p)
(p_12)   edge   [bend right]         node {$\zeta$} (p)
(p_13)   edge      [bend right]       node {$\zeta$} (p)
(p_14)   edge      [bend right]       node {$\zeta$} (p)
(p_16)   edge      [bend right]       node {$\zeta$} (p)
(p)   edge            node {$\tau_{1}$} (p_5)
(p)   edge            node {$\tau_{2}$} (p_6)
(p)   edge             node {$\tau_{1}$} (p_7)
(p)   edge            node {$\tau_{1}$} (p_8)
(p)   edge          node {$\tau_{1}$} (p_11)
(p)   edge          node {$\tau_{2}$} (p_10)
(p)   edge          node {$\tau_{3}$} (p_9)
(p_8)   edge          node {$\tau_{2}$} (p_12)
(p_7)   edge          node {$\tau_{3}$} (p_13)
(p_6)   edge          node {$\tau_{3}$} (p_14)
(p_5)   edge          node {$\tau_{2}$} (p_15)
(p_15)   edge          node {$\tau_{3}$} (p_16)
;  
\end{tikzpicture}

\end{ex}

\begin{rem} 
     \label{rem: q = 3}
The automaton constructed in \Cref{prop:L recognizable} is often not the smallest automaton that 
recognizes the language $\cL$.  Using the minimization technique described in 
 \cite[Theorem 4.26]{HU},  one can obtain an automaton  with fewer states that also recognizes $\cL$.  For example, if $\bc = (1, 1, 1)$, this produces an automaton with only four states: 
 
\begin{tikzpicture}[shorten >=1pt,node distance=2.0cm,auto]
\tikzstyle{every state}=[fill={rgb:black,1;white,10}]
\node[state,initial,accepting]        (p_1)     {$p_1$}; 
\node[state,accepting]           (p_2) [right   of=p_1]     {$p_2$};
\node[state,accepting]           (p_3) [right of=p_2]     {$p_3$};
\node[]           (p_4) [right of=p_3]     {};
\node[state,accepting] (p_111) [below of=p_2] {$p_{\mathbf{1}}$};
\path[->]
(p_1)   edge  [loop above]          node {$\tau_1$} (p_1)
(p_2)   edge  [loop above]         node {$\tau_{2}$} (p_2)
(p_3)   edge  [loop above]          node {$\tau_{3}$} (p_3)
(p_111)   edge  [loop below]          node {$\zeta$} (p_111)
(p_1)   edge            node {$\tau_2$} (p_2)
(p_1)   edge   [bend left]            node {$\tau_3$} (p_3)
(p_2)   edge            node {$\tau_3$} (p_3)
(p_1)   edge [bend right]               node {$\zeta$} (p_111)
(p_3)   edge   [bend left]   node {$\zeta$} (p_111)
(p_2)   edge      [bend left]          node {$\zeta$} (p_111)
(p_111)   edge             node {$\tau_{1}$} (p_1)
(p_111)   edge        [bend left]       node {$\tau_{2}$} (p_2)
(p_111)   edge            node {$\tau_{3}$} (p_3)
;  
\end{tikzpicture}
\end{rem}

In order to relate a language $\cL$ on an alphabet $\Sigma$  to a Hilbert series we need a suitable weight function. Let $T = \K[s_1,\ldots,s_k]$ be a polynomial ring in $k$ variables and denote by $\Mon (T)$ the set of monomials in $T$. A \emph{weight function} is a monoid homomorphism $\rho \colon \Sigma^{\star} \rightarrow \Mon (T)$ such that $\rho(w)=1$ only if $w$ is the empty word. The corresponding generating function is a formal power series in variables $s_1,\ldots,s_k$: 
\[
P_{\mathcal{L},\rho} (s_1,..,s_k) = \sum \limits_{w \in \mathcal{L} } \rho(w). 
\]
We will use the following result (see, e.g., \cite{H} or \cite[Theorem 4.7.2]{St}). 

\begin{thm} 
     \label{thm:rat hilb of lang}
If  $\rho$ is any  weight function on a regular language $\mathcal{L}$ then $P_{\mathcal{L},\rho}$ is a rational function in $\Q(s_1,..,s_k)$. 
\end{thm}

We are ready to establish the ingredient of the proof of \Cref{thm:main} whose proof we had postponed. 

\begin{proof}[Proof of \Cref{prop:second step}]
Since $I_\n = \ker \Phi_n$ and $\Phi_n$ is a homomorphism of degree $q$, we get $R_\n/I_n \cong A_\n$ and, for each $d \in \Z$, 
\[
\dim_{\mathbb{K}} [{R_{\mathbf{n}}\slash I_{\mathbf{n}}}]_{d} = \dim_\K [A_{\n}]_{d q}.  
\]
Recall that the algebra $A_\n$ is generated by monomials. Hence, every graded component has a $\K$-basis consisting of monomials. It follows that $\dim_\K [A_{\n}]_{d q} =\# \Mon ([A_{\mathbf{n}}]_{d q})$. Therefore we get for the equivariant Hilbert series of the filtration $\CI$: 
\[
equivH_{\mathscr{I}}(s_1,..,s_q, t)= 
\sum_{\n \in \N^q} \sum_{d \ge 0} \# \Mon ([A_{\mathbf{n}}]_{dq}) \cdot s^\n t^d, 
\]
where $s^\n = s_1^{n_1} \cdots s_q^{n_q}$ if $\n = (n_1,\ldots,n_q)$. 

Consider now the language $\cL$ described in \Cref{def:language L}. Define a weight function $\rho \colon \Sigma^* \to \Mon(\K[s_1,\ldots,s_q, t])$ by $\rho (\tau_j) = s_j$ and $\rho (\zeta_\bi) = t$ for $\bi \in [\bc]$. Thus, for $w \in \cL$, one obtains $\rho (w) = s^\n t^d$ if $d$ is the number of 
$\zeta$-letters occurring in $w$ and $n_j$ is the number of appearances of $\tau_j$ in $w$. Hence \Cref{prop:bijection} gives that the number of words $w \in \cL_\n$ with $\rho (w) = s^\n t^d$ is precisely $ \# \Mon ([A_{\mathbf{n} + \one}]_{dq})$. Since  $\mathcal{L}$ is the disjoint union of all $\cL_\n$, it follows
\begin{equation} 
      \label{eq:hilb is gen function} 
s_1 \cdots s_q \cdot  equivH_{\mathscr{I}}(s_1,..,s_q, t) = \sum_{\n \in \N_0^q} \sum_{w \in \cL_n} \rho (w) = P_{\cL, \rho} (s_1,..,s_q, t). 
\end{equation}
As the right-hand side is rational by \Cref{thm:rat hilb of lang}, the claim follows. 
\end{proof}
 
\begin{rem}
The method of proof for \Cref{thm:main} is rather general and can also be used in other situations. 
An easy generalization is obtained as follows. Fix $(a_1,\ldots,a_q) \in \N^q$.  For $\n \in \N^q$, consider the homomorphism
\begin{align*}
\widetilde{\Phi}_{\n} \colon R_{\n} = \K[x_{\bi, \bk}  \mid (\bi, \bk) \in [\bc] \times [\n] ] 
& \to \K[y_{j, i_j, k_j} \mid j \in [q], i_j \in c_j, k_j \in [n_j] ] = S_{\n} \\
x_{\bi, \bk} & \longmapsto \prod_{j=1}^q  y_{j, i_j, k_j}^{a_j}, 
\end{align*}
and set $\widetilde{A}_\n = \im \Phi_{\n} = \K \left [ \prod_{j=1}^q  y_{j, i_j, k_j}^{a_j}  \mid  i_j \in c_j, k_j \in [n_j] \right ]$, $\widetilde{I}_\n = \ker \widetilde{\Phi}_\n$. Then $\widetilde{\CI} = \{\widetilde{I}_\n\}_{\n \in \N^q}$ also is an $S_{\infty}^q$-invariant  filtration whose equivariant Hilbert series is rational. 
Indeed, this follows  using the language $\cL$ as above with the following modifications. In the definition of the map $\textbf{m}$ change  Rule (b) to $\widetilde{\textbf{m}} (\zeta_\bi w) = \prod_{j=1}^q  y_{j, i_j, 1}^{a_j}\widetilde{\textbf{m}} (w)$, but keep Rules (a), (c) to obtain a 
map $\widetilde{\textbf{m}} \colon \Sigma^{\star} \rightarrow \Mon (S)$. It induces bijections $\cL_n \to \Mon (\widetilde{A}_{\n + \one})$ as in \Cref{prop:bijection}.  Observe that $[R_\n/\widetilde{I}_\n]_d \cong [\widetilde{A}_\n]_{d a}$, where $a = a_1 + \cdots + a_q$. Thus, using the same weight function $\rho$ as above, we obtain $s_1 \cdots s_q \cdot  equivH_{\mathscr{I}}(s_1,..,s_q, t) = P_{\cL, \rho} (s_1,..,s_q, t)$. 

A systematic study of substantial generalizations will be presented in \cite{MN}. 
\end{rem}



\section{Explicit formulas}

We provide explicit formulas for the Hilbert series of hierarchical models considered in \Cref{thm:main}. 

It is useful to begin by discussing Segre products more generally. To this end we temporarily use some new notation. 

\begin{lemma}
    \label{lem:Segre prod}
Let $ A = \K[a_1,\ldots,a_s] \subset R$  and $B = \K[b_1,\ldots,b_t] \subset S$ be subalgebras of polynomial rings  $R = \K[x_1,\ldots,x_m]$ and  $S = \K[y_1,\ldots,y_n]$ that are generated by monomials $a_1,\ldots,a_s$ of degree $d_1$ and  monomials $b_1,\ldots,b_t$ of degree $d_2$, respectively. 
Let $C$ be the subalgebra of $\K[x_1,\ldots,x_m, y_1,\ldots,y_n]$ that is generated by all monomials $a_i b_j$ with $i \in [s]$ and $j \in [t]$. Using the gradings induced from the corresponding polynomials rings one has, for all $k \in \Z$, 
\[
\dim_{\K} [C]_{k (d_1 + d_2)} = \dim_{\K} [A]_{k d_1} \cdot \dim_{\K} [B]_{k d_2}. 
\]
\end{lemma}

\begin{proof}
This follows from the fact that the non-trivial degree components of the algebras $A, B, C$ have $\K$-bases generated by monomials in the respective algebra generators of suitable degrees. 
\end{proof}

It is customary to consider the algebras occurring in \Cref{lem:Segre prod} as standard graded algebras that are generated in degree one by redefining their grading. In the new gradings, the degree $k$ elements of $A$ are elements that have degree $k d_1$, considered as polynomials in $R$, and similarly the degree $k$ elements of $C$ have degree $k (d_1 + d_2)$ when considered as elements of $\K[x_1,\ldots,x_m, y_1,\ldots,y_n]$. Using this new grading, the statement in the above lemma reads 
\begin{equation}
    \label{eq:hilb Segre}
\dim_{\K} [C]_{d} = \dim_{\K} [A]_{d} \cdot \dim_{\K} [B]_{d}.
\end{equation}
This justifies to call $C$ the \emph{Segre product} of the algebras $A$ and $B$. We denote it by $A \boxtimes B$.


Iterating the above construction we get the following consequence. 

\begin{cor}
          \label{cor:hilb Segre}
Let $A_1,\ldots,A_k$ be subalgebras of polynomial rings and assume every $A_i$ generated by finitely many monomials of degrees $d_i$. Regrade such that every $A_i$ is an algebra that is generated in degree one. Then one has 
\[
\dim_{\K} [A_1 \boxtimes \cdots \boxtimes A_k]_d = \prod_{i = 1}^k \dim_{\K} [A_i]_d. 
\]
\end{cor}

We need an elementary observation. 

\begin{lemma}
      \label{lem:elementary}
Let $\omega \in \C$ be a primitive $k$-th root of unity. If 
\[
f(t) = \sum_{n=0}^{\infty} c_n t^n 
\]
is a formal power series in $t$ with complex coefficients, then 
\[
 \sum_{n=0}^{\infty} c_{k n} x^{k n} = \frac{1}{k} \left [f(t) + f(\omega t) + \cdots + f (\omega^{k-1} t) \right ]. 
\]
\end{lemma}

\begin{proof}
Using geometric sums one gets, for every $n \in \N_0$, 
\[
\sum_{j=0}^{k-1} (\omega^j)^n = \begin{cases}
k & \text{ if $k$ divides $n$} \\
0 & \text{ otherwise.}
\end{cases}
\]
The claim follows. 
\end{proof}

\begin{prop}
      \label{prop:explict computation}
Fix any $q \in \N$ and let $\CI$ be the $S_{\infty}^q$-invariant  filtration considered in \Cref{prop:second step}. 
%
For $j \in [q]$, let $\om_j$ be a $c_j$-th primitive root of unity. Then the equivariant Hilbert series of $\CI$ is 
\begin{align*}
equivH_{\mathscr{I}}(s_1,\dots,s_q, t) 
& = \frac{1}{c_1 \cdots c_q}  \sum_{m_1 \in [c_1],\ldots, m_q \in [c_q]} 
 \frac{\omega_1^{m_1} s_1^{\frac{1}{c_1}} \cdots \omega_q^{m_q} s_q^{\frac{1}{c_q}}}{(1 - \omega_1^{m_1} s_1^{\frac{1}{c_1}})\cdots (1 - \omega_q^{m_q} s_q^{\frac{1}{c_q}}) - t}. 
\end{align*}
\end{prop}

\begin{proof} 
By definition of the map $\Phi_{\cM_{\n}}$, its image is isomorphic to the Segre product of polynomial rings of dimension $c_j n_j$ with $j =1,\ldots,q$. Hence \Cref{cor:hilb Segre}  gives for the equivariant Hilbert series 
\begin{align}
      \label{eq:use of Segre}
equivH_{\mathscr{I}}(s_1,\dots,s_q, t) 
& = \sum_{d\geq 0,\n \in \N^q} \binom{c_1n_1 + d-1}{d} \cdots \binom{c_q  n_q  + d-1}{d}s_1^{n_1}\dots s_q^{n_q}  t^d   \nonumber \\
& = \sum_{d \ge 0} \left \{ \prod_{j = 1}^q \left [ \sum_{n_j \in \N} \binom{c_j n_j + d-1}{d} s_j^{n_j}  \right ] \right  \} t^d 
\end{align}
For any integer $d \ge 0$, one computes
\begin{align*}
\sum_{n \in \N} \binom{n+d-1}{d} s^n = s \sum_{n \in \N_0} \binom{d+n}{n} s^n = \frac{s}{(1 - s)^{d+1}}. 
\end{align*} 
Combined with \Cref{lem:elementary} and using a $c$-th primitive root of unity $\omega \in \C$, we obtain, for any integer $c > 0$, 
\[
\sum_{n \in \N} \binom{c n+d-1}{d} s^n = \frac{1}{c} \sum_{m \in [c]} \frac{\omega^m s^{\frac{1}{c}}}{(1 - \omega^m s^{\frac{1}{c}})^{d+1}}. 
\]
Applying the last formula to the inner sums in Equation \eqref{eq:use of Segre} we get
\begin{align*}
\hspace{4em}&\hspace{-4em}  
equivH_{\mathscr{I}}(s_1,\dots,s_q, t) \\
& = \sum_{d \ge 0} \left \{ \prod_{j = 1}^q \left [ \frac{1}{c_j} \frac{\omega_j^m s_j^{\frac{1}{c_j}}}{(1 - \omega_j^m s_j^{\frac{1}{c_j}})^{d+1}} \right ] \right  \} t^d  \\
& = \sum_{d \ge 0}  \frac{1}{c_1 \cdots c_q}  \left \{ \sum_{m_1 \in [c_1],\ldots, m_q \in [c_q]}  \;
  \frac{\omega_1^{m_1} s_1^{\frac{1}{c_1}}}{(1 - \omega_1^{m_1} s_1^{\frac{1}{c_1}})^{d+1}}  \cdots 
  \frac{\omega_q^{m_q} s_q^{\frac{1}{c_q}}}{(1 - \omega_q^{m_q} s_q^{\frac{1}{c_q}})^{d+1}} 
  \right  \} t^d \\
& =     \frac{1}{c_1 \cdots c_q} \; \sum_{m_1 \in [c_1],\ldots, m_q \in [c_q]} \; 
 \frac{\omega_1^{m_1} s_1^{\frac{1}{c_1}} \cdots \omega_q^{m_q} s_q^{\frac{1}{c_q}}}{(1 - \omega_1^{m_1} s_1^{\frac{1}{c_1}})\cdots (1 - \omega_q^{m_q} s_q^{\frac{1}{c_q}}) - t},   
\end{align*}
as claimed. 
\end{proof}

By \Cref{thm:main}, the above formula for the equivariant Hilbert series can be re-written as a rational function with rational coefficients. 

\begin{ex}
    \label{ex:equiv HS}
(i) Let $c_1 = \cdots = c_q = 1$. Then \Cref{prop:explict computation} gives
\[
equivH_{\mathscr{I}} (s_1,s_2,\dots,s_{q}, t) 
=
\dfrac{s_1\dots s_q}{(1-s_1)\dots(1-s_q)-t}. 
\]
By the argument at the beginning of the proof of \Cref{lem:reduction},  this model has the same equivariant Hilbert series as the corresponding independence model (see \Cref{ex:indep model}). 

(ii) Let $q = c_1 = c_2 = 2$. Then \Cref{prop:explict computation} yields 
\begin{align*}
4 \cdot equivH_{\mathscr{I}} (s_1,s_2, t) 
& =  \frac{\sqrt{s_1 s_2}}{(1-\sqrt{s_1}) (1-\sqrt{s_2}) - t}  - \frac{\sqrt{s_1 s_2}}{(1-\sqrt{s_1}) (1+\sqrt{s_2}) - t}  \\
& \hspace*{.5cm}  - \frac{\sqrt{s_1 s_2}}{(1+\sqrt{s_1}) (1-\sqrt{s_2}) - t}    
+ \frac{\sqrt{s_1 s_2}}{(1+\sqrt{s_1}) (1+\sqrt{s_2}) - t}. 
\end{align*}
Now a straightforward computation gives
\begin{align*}
 equivH_{\mathscr{I}} (s_1,s_2, t) 
&  = \dfrac{s_1s_2(s_1s_2-s_1-s_2-t^2)}{f}, 
\end{align*}
where 
\begin{align*}
f & = s_1 s_2 (s_1 -2)(s_2 -2) + s_1 (s_1 -2) + s_2 (s_2 -2) \\
& \hspace*{.5cm}  - 2t^2(s_1s_2+s_1+s_2)-4t(s_1s_2-s_1-s_2)+(1-t)^4. 
\end{align*}
\end{ex}

\vspace*{.5cm}


There is an alternative method to determine the equivariant Hilbert series whose rationality is 
guaranteed by \Cref{prop:second step}. It directly produces a rational function with rational 
coefficients. This approach applies to any equivariant Hilbert series that is equal to the generating function 
$P_{\cL, \rho}$ determined by a weight function $\rho$ on a regular language $\cL$. 
Indeed, let $\cA = (P, \Sigma, \delta, p_0, F)$ be a finite automaton that recognizes $\cL$. Suppose $P$ has $N$ elements $p_0,\ldots,p_{N-1}$. For every letter $a \in \Sigma$ define a  $0-1$ matrix $M_{\cA, a}$ of size $N \times N$. Its entry at position $(i , j)$ is 1 precisely if there is a transition $\delta (p_j, a) = p_i$. Let $\mathbf{e}_i \in \mathbb{K}^N$ be the canonical basis vector corresponding to   state $p_{i-1}$.  Let $\mathbf{u}=\sum\limits_{p_{i-1} \in F} \mathbf{e}_i \in \mathbb{K}^N$ be the sum of the basis vectors corresponding to the accepting states. Then, for any word $w=w_1\dots w_d$ with $w_i \in \Sigma$, one has 
\[
\mathbf{u}^{T} M_{\cA,w_d}\dots A_{\cA,w_1} \mathbf{e}_1  =\begin{cases} 
      1 & \text{ if } \cA \text{ accepts } w \\
      0 & \text{ if }  \cA \text{ rejects } w.
   \end{cases}
   \]
Let $\rho: \Sigma^*\rightarrow \Mon (\K[s_1,\ldots,s_k])$ be a weight function. Thus, $\rho (w_1 w_2) = \rho (w_1) \cdot \rho (w_2)$  for any $w_1, w_2 \in \Sigma^*$. It follows (see, e.g, \cite[Section 4.7]{St}): 
\begin{align*}
P_{\cL,\rho} (s_1,\dots,s_k) & =\sum_{w\in\cL} \rho(w)  
=  \sum_{d\ge 0} \sum_{w_1,\ldots,w_d \in \Sigma}  \mathbf{u}^T \left( \rho (w_1\dots w_d)  M_{\cA,w_d}\dots A_{\cA,w_1} \right ) \mathbf{e}_1 \\
& =\sum_{d \ge 0} \mathbf{u}^T\left(\sum_{a\in\Sigma} \rho(a) M_{\cA,a} \right)^d \mathbf{e}_1 
= \mathbf{u}^T \left( \id_{N}-\sum_{a\in\Sigma} \rho(a) M_{\cA,a} \right)^{-1}\mathbf{e}_1. 
\end{align*}
Thus, the generating function $P_{\cL,\rho} (s_1,\dots,s_k)$ is rational with rational coefficients and can be explicitly computed from the  automaton $\cA$ using linear algebra.  

In the proof of \Cref {prop:second step}, we showed (see Equation \eqref{eq:hilb is gen function}) that the equivariant Hilbert series of a considered filtration is, up to a degree shift, equal to a generating function. Hence, the above approach can be used to compute directly this Hilbert series as a rational function with rational coefficients. We implemented  the resulting algorithm in Macaulay2 \cite{GS}. It is posted at {\tt \href{http://www.sites.google.com/view/aidamaraj/home/research}{\color{black}{http://www.sites.google.com/view/aidamaraj/home/research}}}. 

\begin{ex} 
In \Cref{prop:second step}, consider the case where $\bc=(1,1\dots,1) \in \N^q$. The automaton constructed in \Cref{prop:L recognizable} can be reduced to one with only $q+1$ states (see \Cref{rem: q = 3} if $q = 3$): 

\begin{tikzpicture}[shorten >=1pt,node distance=2.5cm,auto]
  \tikzstyle{every state}=[fill={rgb:black,1;white,10}]

  \node[state,initial,accepting]        (p_1)     {$p_1$}; 
   \node[state,accepting]           (p_2) [right   of=p_1]     {$p_2$};
   \node[]           (p_.) [right of=p_2]     {$\dots$};
  \node[state,accepting]           (p_3) [right of=p_.]     {$p_q$};
    \node[]           (p_4) [right of=p_3]     {};
  \node[state,accepting] (p_111) [below of=p_2] 
  {$p_{\one}$};
  \path[->]
   (p_1)   edge  [loop above]          node {$\tau_1$} (p_1)
     (p_2)   edge  [loop above]         node {$\tau_{2}$} (p_2)
      (p_3)   edge  [loop above]          node {$\tau_{q}$} (p_3)
       (p_111)   edge  [loop below]          node {$\zeta$} (p_111)
  (p_1)   edge            node {$\tau_2$} (p_2)
   (p_1)   edge   [bend left]            node {$\tau_q$} (p_3)
    (p_2)   edge       [bend left]      node {$\tau_q$} (p_3)
   (p_1)   edge               node {$\zeta$} (p_111)
    (p_3)   edge   [bend left]   node {$\zeta$} (p_111)
    (p_2)   edge      [bend left]          node {$\zeta$} (p_111)
       (p_111)   edge     [bend left]         node {$\tau_{1}$} (p_1)
      (p_111)   edge              node {$\tau_{2}$} (p_2)
      (p_111)   edge            node {$\tau_{q}$} (p_3)
    ;      
\end{tikzpicture}

Hence, listing $p_{\one}$ as the last state we obtain for the equivariant Hilbert series of the filtraction $\CI$: 
\begin{align*}
\hspace{1em}&\hspace{-1em}  
{equivH}_{\mathscr{I}}(s_1,\ldots,s_q,t)  \\
& = s_1s_2\cdots s_q \cdot  \mathbf{u}^T \left (\id_{q+1}-\sum_{a \in\Sigma} \rho (a)  M_{\cA,w} \right )^{-1} \mathbf{e}_1\\
& = s_1s_2 \cdots s_q \begin{bmatrix}
1   \\
     1   \\
     \vdots   \\     
      1  
\end{bmatrix}^T
\begin{bmatrix}
1- s_1 & 0 & 0 & \ldots & 0 & 0 & -s_1 \\
-s_2 & 1 - s_2 & 0 & \ldots & 0  & 0 & -s_2 \\
-s_3 & -s_3 & 1 - s_3 & \ldots & 0 & 0 & -s_3 \\
 \vdots  &\vdots &  \vdots  &      \ddots &  \vdots & \vdots  &  \vdots \\ 
 -s_{q-1} & -s_{q-1} & -s_{q-1} & \ldots &  1-s_{q-1} & 0 &  -s_{q-1} \\
 - s_q & - s_q & - s_q & \ldots & -s_q & 1 - s_q & - s_q \\
  -t   &  -t  & -t  & \dots & -t & -t  &1 - t
\end{bmatrix}^{-1}
\begin{bmatrix}
   1   \\
     0   \\
     \vdots\\
      0   
\end{bmatrix} \\
&= \dfrac{s_1\cdots s_q}{(1-s_1) \cdots(1-s_q)-t}, 
\end{align*}
where the first column of the  inverse matrix can be determined using suitable minors. 
Of course, the result is the same as in 
\Cref{ex:equiv HS}. 
\end{ex}



\end{document}